\documentclass[11pt]{article}
\usepackage[utf8]{inputenc}
\usepackage[T1]{fontenc}
\usepackage{lmodern}
\usepackage[english]{babel}
\usepackage{microtype}
\usepackage{a4wide}

\usepackage{amsmath,amssymb,amsfonts,amsthm}
\usepackage{mathtools,accents}
\usepackage{mathrsfs}
\usepackage{aliascnt}
\usepackage{braket}
\usepackage{bm}
\usepackage{physics}
\usepackage{tikz-cd}

\usepackage[a4paper,margin=3cm]{geometry}
\usepackage[citecolor=blue,colorlinks]{hyperref}

\usepackage{enumerate}
\usepackage{xcolor}

\makeatletter
\def\newaliasedtheorem#1[#2]#3{
  \newaliascnt{#1@alt}{#2}
  \newtheorem{#1}[#1@alt]{#3}
  \expandafter\newcommand\csname #1@altname\endcsname{#3}
}
\makeatother

\numberwithin{equation}{section}

\newtheoremstyle{slanted}{\topsep}{\topsep}{\slshape}{}{\bfseries}{.}{.5em}{}

\theoremstyle{plain}
\newtheorem{theorem}{Theorem}[section]
\newaliasedtheorem{proposition}[theorem]{Proposition}
\newaliasedtheorem{lemma}[theorem]{Lemma}
\newaliasedtheorem{corollary}[theorem]{Corollary}
\newaliasedtheorem{conjecture}[theorem]{Conjecture}
\newaliasedtheorem{counterexample}[theorem]{Counterexample}

\theoremstyle{definition}
\newaliasedtheorem{definition}[theorem]{Definition}
\newaliasedtheorem{question}[theorem]{Question}
\newaliasedtheorem{example}[theorem]{Example}

\theoremstyle{remark}
\newaliasedtheorem{remark}[theorem]{Remark}

\newcommand{\setN}{\mathbb{N}}

\newcommand{\setR}{\mathbb{R}}

\newcommand{\Z}{\mathbb{Z}}
\renewcommand{\subset}{\subseteq}

\newcommand{\defeq}{\mathrel{\mathop:}=}

\newcommand{\mres}{\mathbin{\vrule height 1.6ex depth 0pt width 0.13ex\vrule height 0.13ex depth 0pt width 1.3ex}}

\newcommand{\eps}{\varepsilon}

\let\altphi\phi
\let\phi\varphi
\let\varphi\altphi
\let\altphi\undefined

\DeclareMathOperator{\vol}{\mathrm{vol}}

\DeclareMathOperator{\supp}{supp}

\newcommand{\haus}{\mathcal{H}}
\newcommand{\leb}{\mathcal{L}}

\newcommand{\dist}{\mathsf{d}}

\newcommand{\meas}{\mathfrak{m}}

\newcommand{\hess}{{\mathrm{Hess}}}

\DeclareMathOperator{\RCD}{RCD}
\DeclareMathOperator{\ncRCD}{ncRCD}

\DeclareMathOperator{\CD}{CD}

\newfont{\tmpf}{cmsy10 scaled 2500}

\newcommand{\R}{\setR}							
\newcommand{\N}{\setN}							

\begin{document}

\title{On the intrinsic and extrinsic boundary for metric measure spaces with lower curvature bounds 
}

\author{
 Vitali Kapovitch
\thanks{University of Toronto, \url{vtk@math.toronto.edu}, V.K.  is partially supported by a Discovery grant from NSERC }
\and 
Xingyu Zhu
\thanks{Fields Institute, \url{xzhu@fields.utoronto.ca}} 
}
\date{\today}

\maketitle

\begin{abstract}
    We show that if an Alexandrov space $X$ has an Alexandrov subspace $\bar\Omega$ of the same dimension disjoint from the boundary of $X$, then the topological boundary of $\bar\Omega$ coincides with its Alexandrov boundary. Similarly, if a noncollapsed $\RCD(K,N)$ space $X$ has a noncollapsed $\RCD(K,N)$ subspace $\bar\Omega$ disjoint from boundary of $X$ and with mild boundary condition, then the topological boundary of $\bar\Omega$ coincides with its De Philippis-Gigli boundary. We then discuss some consequences about convexity of such type of equivalence.
\end{abstract}

\section{Introduction}
The intrinsic notion of boundary has been extensively studied for both noncollapsed $\RCD(K,N)$ spaces ($\ncRCD(K,N)$ in short) and Alexandrov spaces. When we say Alexandrov spaces, we always mean complete, geodesic, finite dimensional Alexandrov space. For an Alexandrov space $(A,\dist_A)$, { Burago, Gromov and  Perelman introduced the definition of boundary in \cite{BGP92}}, deonted by $\mathcal{F}A$, see \eqref{eq:Alexboundary}. {From the uniqueness of tangent cones along in interiors of geodesics proved by  Petrunin }in \cite{Petrunin98}, it can be deduced that the interior of an Alexandrov space, i.e.\ $A\setminus\mathcal{F}A$, is strongly convex, which means that  any geodesic joining points in the interior does not intersect $\mathcal{F}A$. For a $\ncRCD(K,N)$ space $(X,\dist,\haus^N)$, there are $2$ intrinsic definitions of boundary. One is defined by Kapovitch-Mondino in \cite{KapMon19}, in the same spirit of defining the boundary for an Alexandrov space, we denote this boundary also by $\mathcal{F}X$, see \eqref{eq:KMboundary}. The other is defined by De Philippis-Gigli in \cite{DPG17}, making use of the stratification of the singular set. We denote this boundary by $\partial X$, see \eqref{eq:DPGboundary}. 

In parallel to the strong convexity of the interior of an Alxeandrov space,  it is conjectured by De Philippis and Gigli \cite[Remark 3.8]{DPG17} that the interior of $X$, i.e.\ $X\setminus\partial X$ is strongly convex.  We will see that this conjecture follows from the conjecture that the two notions of the boundary of $\ncRCD(K,N)$ spaces  agree.


In this paper, we look at the boundary from an an extrinsic point of view, namely, given $K\in \R$ and positive integer $N$  we consider two sitiations

\begin{enumerate}
    \item\label{item:alex} an $N$-dimensional Alexandrov space has an $N$-dimensional Alexandrov subspace;
    \item \label{item:RCD} a $\ncRCD(K,N)$ space has a $\ncRCD(K,N)$ subspace with mild boundary control.
\end{enumerate}  

We prove that in the case of \eqref{item:alex}  the intrinsic boundary of an Alexandrov subspace coincides with the topological boundary, and in the case of \eqref{item:RCD} the De Phlippis-Gigli boundary coincides with the topological boundary. See the precise statement in Theorem \ref{thm:main1} and Theorem \ref{thm:main2} below. A direct consequence is  that  synthetic curvature bounds on a subspace automatically imply regularity of its topological boundary, for example topological structure and rectifiability, see \cite{BNS20}.

\begin{theorem}\label{thm:main1}
Let $(X,\dist_X)$ be an $N$-dimensional Alexandrov space, $N\in \setN$, and $\Omega\subset X$ be open, if $(\bar{\Omega},\dist_\Omega)$ is an Alexandrov space, $\bar{\Omega}\cap\mathcal{F}X=\varnothing$, and $\Omega={\rm Int}_{\rm top}(\bar{\Omega})$, where ${\rm Int}_{\rm top}(\bar\Omega)$ is the topological interior, i.e. the largest open subset of of $\bar\Omega$, 
then 
\begin{enumerate}
    \item\label{main1:item1} $\partial_{\rm top}\bar\Omega=\mathcal{F}\bar\Omega$;
    \item\label{main1:item2} any (minimizing) geodesic in $(\bar\Omega,\dist_\Omega)$ joining two points in $\Omega$ is a local geodesic of $(X,\dist_X)$;

    \item\label{main1:item3} any (minimizing) geodesic in $(\bar{\Omega},\dist_\Omega)$ is a quasi-geodesic in  $(X,\dist_X)$.
\end{enumerate}
\end{theorem}
Theorem \ref{thm:main1} will follow from the invariance of domain theorem for Alexandrov spaces, Theorem \ref{thm:inv}. {
The proof of Theorem \ref{thm:inv} has been worked out on the mathoverflow website quite a while ago by Belegradek, Petrunin and Ivanov but does not  seem to exist in literature. Since we have an application of this theorem we present the proof following closely the existing one by Belegradek-Petrunin-Ivanov.}  It works in a more general purely topological category of MCS spaces, see Theorem~\ref{thm:inv-dom-mcs}.

\begin{remark}\label{rem:counterex}
The assumption $\Omega={\rm Int}_{\rm top}(\bar \Omega)$ is clearly necessary and cannot be removed. For example, let $X=\R^n$, $\Omega=\R^n\setminus\{0\}$, which is open and dense. We see that $\bar\Omega=X$ is an Alexandrov space without Alexandrov boundary, but the topological boundary of is $\{0\}$ which is not empty. This shows that item \ref{main1:item1} does not hold without the assumption $\Omega={\rm Int}_{\rm top}(\bar \Omega)$  even for smooth manifolds. 

Next, the assumption that $\bar{\Omega}\cap\mathcal{F}X=\varnothing$ is also clearly necessary. For example let $X$ be the closed unit disk in $\R^2$ and $\Omega=X$. Then $\partial \bar \Omega$ is empty while $\mathcal{F}\bar\Omega=\mathbb S^1$.


 Also, a geodesic in $(\bar{\Omega},\dist_\Omega)$ joining two points on the boundary need not to be a local geodesic in $(X,\dist_X)$, so the conclusion in item \ref{main1:item3} of Theorem \ref{thm:main1} that a geodesic in $(\bar{\Omega},\dist_\Omega)$ is a quasi-geodesic in the ambient space is optimal.

  For example, consider the space $X\defeq D^2\times\{0\}\sqcup_{\mathbb{S}^1\times\{0\}}\mathbb{S}^1\times [0,\infty)$ with length metric, which is a cylinder glued along the boundary circle with a disk at the bottom, this is an Alexandrov space of non-negative curvature. Then let $\Omega=\mathbb{S}^1\times (0,\infty)$. Clearly $\Omega$ is open, however, any geodesic on $\partial_{\rm top}\Omega=\mathcal{F}\Omega=\mathbb{S}^1\times\{0\}$, which is an arc, is never a local geodesic w.r.t. the metric of $X$, since a segment in $D^2$ connecting any 2 points on its boundary circle is always shorter than the corresponding arcs.
This example also shows that $\bar{\Omega}$ need not to be locally convex. Compare this with Theorem \ref{thm:han}.

\end{remark}

For $\ncRCD(K,N)$ spaces, we are able to obtain a similar result to Theorem \ref{thm:main1} under an extra  assumption of a local Lipschitz condition on the metric $\dist_\Omega$ which serves as a weak substitute for the regularity of the topological boundary.
 \begin{theorem}\label{thm:main2}
     Let $(X,\dist,\haus^N_X)$ be a $\ncRCD(K,N)$ space, $\Omega$ be an open subset of $X$ such that $\Omega={\rm Int}_{\rm top}(\bar \Omega)$ and $\bar \Omega\cap \partial X=\varnothing$. Suppose that $(\bar{\Omega},\dist_\Omega,\haus^N_{\bar\Omega})$ is also an $\RCD(K,N)$ space and for every $x\in\partial_{\rm top}\bar{\Omega}$ there exist an neighborhood $U_x$ of $x$ and $C(U_x)>1$ such that $\dist_\Omega\le C(U_x) \dist_X$ when restricted to $U_x\cap\bar \Omega$. Then $\partial_{\rm top}\bar \Omega=\partial \bar\Omega$.
 \end{theorem}
 Here, $\haus^N_X$ (resp. $\haus^N_{\bar\Omega})$) is the Hausdorff measure induced by $\dist_X$ (resp. $\dist_{\Omega}$). Notice the following relations between the two Hausdorff measures:
 
 \begin{remark}\label{rmk:equiHaus}
  From our assumption and the definition of intrinsic length metric, it follows that $\bar\Omega$ is embedded in $X$ in a locally biLipschitz way, i.e.\  for any $x\in \partial_{\rm top}\bar\Omega$ and its neighborhood $U_x$, $\dist_X\le \dist_{\Omega}\le C \dist_{X}$ when restricted to $U\cap\bar\Omega$, so the notions such as Hausdorff dimension and measure zero sets for both Hausdorff measures are equivalent for sets in $\bar\Omega$ {since we can always find a countable covering by neighborhoods on which the 2 metrics are biLipschitz to each other}.
  

 \end{remark}
 
There are 2 main technical difficulties in proving Theorem \ref{thm:main2}. The first being that in general there is no topological information on any neighborhood of a singular point. An important fact used to prove the invariance of domain theorem for Alexandrov spaces is that every point has a neighborhood homeomorphic to a cone over its space of directions, which is not available for $\ncRCD(K,N)$ spaces. In particular, as opposed to the situation in Alexandrov spaces, for a given point in an $\ncRCD(K,N)$ space  its tangent cone(s)in general  do not carry topological information of its neighborhood. For example,  Colding-Naber \cite{CN11} constructed an example of a noncollapsed Ricci limit space with a singular point at which there are two non-homeomorphic tangent cones. Another difficulty is that the topological boundary may  in principle vanish when taking tangent cones.  Conjecturally this cannot happen but this is unknown at the moment. 
A model case of this phenomenon would be a cusp, for example $X=\R^2$, and $\Omega=\{(x,y)\in \R^2: y<\sqrt{|x|}\}$, where $0\in \partial_{\rm top}\bar\Omega$ but its tangent cone in $\bar\Omega$ and in $X$ are both $\R^2$. We can quickly rule out this case since if $\bar\Omega$ were a $\ncRCD(K,N)$ space, then $0$ would have density $1$ in $\bar\Omega$, which in turn implies the neighborhood of $0$ $\Omega$ is a manifold, a contradiction. However, this argument does not work if the point on the topological boundary is itself a singular point of the ambient space. A unified way to overcome both difficulties is to find a regular point on the topological boundary, if it is more than De Philippis-Gigli boundary. Indeed, we are able to do this with the help of Deng's H\"older continuity of tangent cones along the interior of a geodesic, \cite{deng2020holder}.

A motivation for studying the extrinsic notion of boundary is provided by the following observation on manifolds. Han in \cite{han20} showed that for a weighted $n$-dimensional manifold $(M,g, e^{-f}\vol_g)$ with smooth boundary, the measure valued Ricci tensor 
\begin{equation}
    \mathrm{\bf Ric}(\nabla \phi,\nabla \phi):=\mathbf{\Delta}\frac{|\nabla \phi|^2}{2}-(\langle\nabla \phi,\nabla \Delta \phi\rangle+|\hess_{\phi}|^2)e^{-f}\vol_g
\end{equation}
 defined by Gigli \cite{Gigli14} can be expressed as
\begin{equation}
    \mathrm{\mathbf{Ric}}=(\mathrm{Ric}+\hess_f) e^{-f}\vol_g+ \mathrm{II}_{\partial M}e^{-f}\haus^{n-1}|_{\partial M},
\end{equation}
 where $\mathbf{\Delta}$ is the measure valued Laplacian. If $(M,g, e^{-f}\vol_g)$ satisfies $\CD(K,\infty)$ condition, then $\mathrm{\bf Ric}\ge K e^{-f}\vol_g$. Combined with Han's expression, this lower bound in particular implies that the second fundamental form is non-negative definite, which means the boundary is convex and it is well known that this implies that geodesics joining interior points do not intersect boundary. Han further interprets this convexity where a subset and its topological boundary are considered, moreover, the boundary is not $C^2$ so it is not possible to define the second fundamental form on it. 
 To proceed, we fix some notations. For a length metric space $(X,\dist)$, and an open connected subset $\Omega\subset X$, denote by $\dist_{\Omega}$ the intrinsic length metric on $\Omega$, it extends by continuity to $\bar\Omega$. Denote by $\partial_{\rm top}\bar\Omega$ the topological boundary of $\bar\Omega$ in $X$. More precisely, Han proved

\begin{theorem}[\cite{han20}]\label{thm:han}
Let $(M,g)$ be a complete $n$-dimensional manifold, and $\Omega\subset M$ be open. Suppose that $(\bar{\Omega}, \dist_\Omega, \meas)$ satisfies that $\supp(\meas)=\bar{\Omega}$ and $\CD(K,\infty)$ condition, then $\meas\mres \Omega\ll \vol_g\mres \Omega$, if furthermore $\bar\Omega$ has Lipschitz and $\haus^{n-1}$-a.e.\ $C^2$ boundary, then $\meas(\partial_{\rm top}\Omega)=0$ and $(\bar{\Omega},\dist_\Omega)$ is {locally convex}, i.e., every (minimizing) geodesic in $(\bar{\Omega},\dist_\Omega)$ is a local geodesic in $(M,g)$.
\end{theorem}

In particular, every minimizing geodesic in $(\bar{\Omega},\dist_\Omega)$ joining $2$ points in $\Omega$ does not intersect ${\partial_{\rm top}\Omega}$. We would like to generalize to non-smooth setting the {above theorem of Han, but in view of Remark \ref{rem:counterex}, it is not true that the (synthetic) Ricci curvature lower bound on a closed subset forces the set to be locally convex. The correct notion to consider for metric spaces is the locally totally geodesic property. 

\begin{definition}
	Let $(X,\dist)$ be a geodesic metric space. A { connected open subset $\Omega$} is said to be \emph{locally totally geodesic} if every (minimizing) geodesic {in $(\bar \Omega, \dist_{\Omega})$} joining two points in $\Omega$ is a local geodesic in $(X,\dist)$.
\end{definition}
} 


{With this notion, we see from item \ref{main1:item2} of Theorem \ref{thm:main1} that we have shown that the synthetic sectional curvature lower bound on the closure of an open subset forces this open subset to be locally totally geodesic.

 For $\ncRCD$ spaces, the natural approach to generalize the fact that Ricci curvature lower bound on a subset forces locally totally geodesic property is  to show the equivalence between the intrinsic and topological boundary, since the convexity results for intrinsic boundary will then apply to the topological boundary as well. For example, with extra assumption that Kapovitch-Mondino boundary and De Philippis-Gigli boundary coincide, we can derive that the interior of an $\ncRCD(K,N)$ subspace is locally totally geodesic by combining Theorem \ref{thm:main2} and Theorem \ref{thm:intconv}. See also Corollary \ref{cor:loc-total-geo}. }   

 However, for $\ncRCD(K,N)$ spaces, the strong convexity of its (intrinsic) interior is not presently known, to derive it we need an extra assumption that the Kapovitch-Mondino boundary and the De Philippis-Gigli boundary are the same.

\begin{theorem}[Corollary \ref{cor:intconv}]\label{thm:intconv}
 Let $(X,\dist,\haus^N)$ be a $\ncRCD(K,N)$ space. Assume $\partial X=\mathcal{F} X$, then ${\rm Int}(X)\defeq X\setminus \partial X$ is strongly convex, i.e.\ any geodesic joining points in ${\rm Int}(X)$ does not intersect $\partial X$.
\end{theorem}

  Although the equivalence between the two boundary notions, hence the strong convexity of the interior of $\ncRCD(K,N)$ space is unknown, we can still obtain an a.e.\ version of convexity of the interior of a $\ncRCD(K,N)$ space. This in turn implies that  for a $\ncRCD(K,N)$ subset, intrinsic geodesics joining most interior points are away from its topological boundary. The a.e.\ convexity of interior follows from the following more general a.e.\ convexity of regular set at essential dimension which is a slight generalization of pairwise  a.e. convexity of $\mathcal{R}_n$ proved by Deng \cite[Theorem 6.5]{deng2020holder}.
  
\begin{proposition}\label{thm:almostconvex}
Let $(X,\dist,\meas)$ be an $\RCD(K,N)$ space of essential dimension $n$. For \textit{every} $x\in X$, there exists a subset $R_x\subset \mathcal{R}_n$ so that $\meas(X\setminus R_x)=0$ and for any $y\in R_x$ there is a minimizing geodesic joining $x,y$ contained in $\mathcal{R}_n$ except possibly for $x$.  
\end{proposition}

For the proof we need the technique of localization via transport rays of any $1$-Lipschitz function, developed by Cavalletti-Mondino \cite{CavMon15} in non-smooth setting. 

Finally, we conjecture that Theorem \ref{thm:han} holds in much larger generality including the measure regularity part, see Conjecture \ref{conj:collapseconv}.

The paper is organized as follows: In section \ref{sec:prelim}, we recall concisely the structure results for Alexandrov and $\RCD(K,N)$ spaces. In section \ref{sec:inv} we prove invariance of domain theorem for Alexandrov spaces. Section \ref{sec:equi} is devoted to the proof of main theorems Theorem \ref{thm:main1} and Theorem \ref{thm:main2}. The last two sections, section \ref{sec:AppConv} and \ref{sec:almost} focus on applications of the main theorems to subsets satisfying $\ncRCD(K,N)$ condition in various ambient spaces.

\smallskip\noindent
\textbf{Acknowledgement.} The second named author thanks Anton Petrunin for bringing invariance of domain for Alexandrov spaces to his attention, Qin Deng for suggesting Proposition \ref{thm:almostconvex}, Igor Belegradek and Jikang Wang for several helpful discussions. 

\section{Preliminary}\label{sec:prelim}

\subsection{Stratified spaces}\label{subsec:} 
In this section we give a brief review of topological stratified spaces. 

\begin{definition}
A metrizable space $X$ is called an \emph{MCS-space (space with multiple conic singularities)} of dimension $n$ if every point $x\in X$ has a neighborhood pointed homeomorphic to the open cone over a compact  $(n-1)$-dimensional MCS space. Here we assume the empty set to be the unique $(-1)$-dimensional MCS-space.
\end{definition}

\begin{remark}
A compact $0$-dimensional MCS-space is a finite collection of points with discrete topology. A 1-dimensional MCS-space is a locally finite graph.
\end{remark}

An open conical neighborhood of a point in an MCS-space is unique up to pointed homeomorphism~\cite{Kwun}. However given an open conical neighborhood $U$ of $x\in X$ pointed homeomorphic to a cone over an $(n-1)$-dimensional space $\Sigma_x$, the space $\Sigma_x$  need not be uniquely determined by $U$.

It easily follows from the definition that an MCS space has a natural topological stratification constructed as follows.

We say that a point $p\in X$ belongs to the $l$-dimensional stratum $X_l$ if $l$ is the maximal number $m$ such that the conical neighbourhood 
of $p$ is pointed homeomorphic to $\R^m\times K(S)$ for some MCS-space $S$. It is clear that $X_l$ is an $l$-dimensional topological manifold. It is also immediate that for $x\in X_l$ all points in the conical neighborhood of $X$ belong to the union of $X_k$ with $k\ge l$. Therefore the closure $\bar X_l$ of the $l$-stratum is contained in the union $\cup_{m\le l} X_m$ of the strata of dimension at most $l$.

The $n$ stratum $X_n$ is an $n$-dimensional manifold and by above it is open and dense in $X$. We will also refer to $X_n$ as the \emph{top stratum} of $X$.

\subsection{Structure theory for $\RCD(K,N)$ spaces}\label{subsec:ncRCD}
When writing $\RCD(K,N)$ space, we always assume that $N\in [1,\infty)$. 
We assume familiarity with the structure theory of $\RCD(K,N)$ spaces and just collect a few facts to fix notations. 

\begin{definition}
    Given an $\RCD(K,N)$ space $(X,\dist,\meas)$, let $\mathcal{R}_k$ be the set of points at which the tangent cone is $(\R^k,|\cdot|,\leb^k)$, for $k\in [1,N]\cap \setN$. $\mathcal{R}(X)\defeq \cup_k \mathcal{R}_k$ is called the regular set of $X$. 
\end{definition}

If there is no confusion we also write $\mathcal{R}$ instead of $\mathcal{R}(X)$. It is shown in \cite{MN14} that $\meas(X\setminus\cup_{k}\mathcal{R}_k)=0$ and each $\mathcal{R}_k$ is $\haus^k$-rectifiable. Then it is shown in \cite{BrueSemola20Constancy} that there is a unique $n\in [1,N]\cap \setN$ such that $\meas(X\setminus\mathcal{R}_n)=0$. Such $n$ is called the essential dimension of $(X,\dist,\meas)$ which is also denoted by ${\rm essdim}$. 
{It is equal to the maximal $k$ such that $\mathcal{R}_k$ is non empty, see for example \cite{kitabeppu2017sufficient}.}
The singular set $\mathcal{S}$ is the complement of the regular set, $\mathcal{S}\defeq X\setminus \cup_{k}\mathcal{R}_k$.  { The singular set has measure zero.}

The notion of noncollapsed $\RCD(K,N)$ ($\ncRCD(K,N)$ in short) is proposed in \cite{DPG17}, requiring that $\meas=\haus^N$, which in turn implies $N\in\setN$ and the essential dimension of a $\ncRCD(K,N)$ space is exactly $N$, see \cite[Theorem 1.12]{DPG17}. When considering $\ncRCD(K,N)$ spaces, finer structure results are available.

The density function
\begin{equation}
    \Theta_N(x)\defeq\lim_{r\to 0}\frac{\haus^N(B_r(x))}{\omega_N r^N}\le 1
\end{equation}
plays a crucial role in the study of regularity of $\ncRCD(K,N)$ spaces. The existence of the  limit and the  upper bound $1$ come from the Bishop-Gromov inequality. Note that the density function characterizes the regular points in the following way \cite[Corollary 1.7]{DPG17}:
\begin{equation}
    \Theta_N(x)=1 \Leftrightarrow x\in \mathcal{R}_N=\mathcal{R}.
\end{equation}

  Thanks to the splitting theorem \cite{Gigli13} and the volume cone to metric cone property \cite{DPG16} in a $\ncRCD(K,N)$ space, the singular set $\mathcal{S}$ is stratified into 
\[
\mathcal{S}_0\subset \mathcal{S}_1\subset \cdots\subset \mathcal{S}_{N-1},
\]
where for $0\le k\le N-1$, $k\in \N$, $\mathcal{S}_k=\{x\in \mathcal{S}: \text{no tangent cone at $x$ is isometric to } \R^{k+1}\times C(Z)\text{ for any metric space } Z\}$, where $C(Z)$ is the metric measure cone over a metric space $Z$. It is proved in \cite[Theorem 1.8]{DPG17} that
\begin{equation}\label{eq:sing}
    \dim_{\haus}(\mathcal{S}_k)\le k.
\end{equation}
 With the help of the metric Reifenberg theorem \cite[Theorem A.1.1-A.1.3]{Cheeger-Colding97I}, it can be derived that for points whose the density is close to $1$ there is a neighborhood homeomorphic to a smooth manifold. We have from \cite[Theorem 1.7, Corollary 2.14]{KapMon19} that

\begin{theorem}\label{thm:regular}
    Let $(X,\dist,\meas)$ be a $\ncRCD(K,N)$ space, and $\alpha\in (0,1)$. There exists $\delta\defeq \delta(\alpha,K,N)>0$ small enough so that if $x\in X$ satisfies $\Theta_N(x)> 1-\delta$, then there is a neighborhood of $x$ biH\"older homeomorphic to a smooth manifold with H\"older exponent $\alpha$. Moreover the set $\{x\in X: \Theta_N(x)> 1-\delta\}$ is open and dense.
\end{theorem}

We call such points manifold points, and call the complement non-manifold points. It then follows that the set of non-manifold points has Hausdorff codimension at least $1$ since it is contained in $S^{N-1}$. 
 
 Finally let us recall here some facts about the boundary of a $\ncRCD$ space $(X,\dist,\haus^N)$. Based on the stratification of $\mathcal{S}$, De Philippis and Gigli proposed the following definition of the boundary of a $\ncRCD(K,N)$ space $(X,\dist,\meas)$:
\begin{equation}\label{eq:DPGboundary}
    \partial X\defeq \overline{\mathcal{S}_{n-1}\setminus \mathcal{S}_{n-2}}.
\end{equation}
On the other hand, Kapovitch-Mondino (\cite{KapMon19}) proposed another recursive definition of the boundary analogous to that of Alexandrov spaces, for $N\ge 2$:  
\begin{equation}\label{eq:KMboundary}
    \mathcal{F}X\defeq\{x\in X:  \exists Y\in {\rm Tan}(X,\dist,\meas,x), Y=C(Z), \mathcal{F}Z\neq \varnothing\}.
\end{equation}
In this definition $Z$ must be a non-collapsed $\RCD(N-2,N-1)$ space with suitable metric and measure (\cite[Lemma 4.1]{KapMon19}, after \cite{Ketterer2015}), so one can inductively reduce the consideration to the case $N=1$, in which case the classification is completed in \cite{KL15}.

The measure theoretical and topological structure of De Philippis-Gigli's boundary is subsequently studied in \cite{BNS20} and \cite{BPS21}.  We will need the following relation from combining \cite[Lemma 4.6]{KapMon19} and \cite[Theorem 6.6]{BNS20}: 
\begin{equation}\label{eq:boundaryrelation}
    \mathcal{S}^{N-1}\setminus\mathcal{S}^{N-2}\subset\mathcal{F}X\subset \partial X.
\end{equation}
 An implication of the above relation is that not having boundary in both senses are the same, {  and is equivalent to $ \mathcal{S}^{N-1}\setminus\mathcal{S}^{N-2}=\varnothing$.} It is conjectured that $\mathcal{F}X=\partial X$, and this is verified for Alexandrov spaces and Ricci limit spaces with boundary, see \cite[Chapter 7]{BNS20}.

 \subsection{Structure theory of Alexandrov Spaces}\label{subsec:Alexandrov}
Observe that the structure theory of $\ncRCD(K,N)$ spaces holds for Alexandrov spaces since $N$-dimensional Alexandrov spaces with lower curvature bound $K$ are $\ncRCD(K,N)$ spaces \cite{Petrunin11}, though some results can have different, usually easier, proofs. Instead of attempting to give a thorough introduction, we collect here the following facts that are necessary for this paper and are more refined than that of $\ncRCD(K,N)$ spaces. We refer readers to \cite{BGP92, BBI01, Petr-conv} for detailed structure theory of Alexandrov spaces. 

Fix an $N$-dimensional Alexandrov space $(X,\dist)$. We describe the  tangent cones, boundary and topological structure of $X$.
 
Tangent cones in an Alexandrov space are nicer than those in $\ncRCD$ spaces, for example, the tangent cone at every point is unique. To better describe tangent cones, we introduce the space of directions:

\begin{definition}
For any $p\in X$, we say that any $2$ geodesics emanating from $p$ have the same direction if their angle at $p$ is zero. This induces an equivalence relation on the space of all geodesics emanating from $p$ and the angle induces a metric on the space of equivalent classes of such geodesics. The metric completion of it is the space of directions at $p$, denoted by $\Sigma_p(X)$.
\end{definition} 

$\Sigma_p(X)$ is an $(N-1)$-dimensional Alexandrov space of curvature lower bound $1$ \cite[Theorem 10.8.6]{BBI01}. The (metric) tangent cone at $p$ is the metric cone over $\Sigma_p(X)$, this definition is consistent with the (blow-up) tangent cone $T_pX$ obtained by taking the pGH limit of $(X, r^{-1}\dist,p)$ as $r\to 0$. This observation along with Perelman's stability theorem \cite{Perelman91} implies that $p$ has a neighborhood homeomorphic to a cone over $\Sigma_p(X)$, therefore $X$ is an n-dimensional MCS-space by induction. For an alternative proof of this result see \cite{Per-Morse}. 

The boundary $\mathcal{F}X$ is defined for $N\ge 2$ as 
\begin{equation}\label{eq:Alexboundary}
    \mathcal{F}X=\{p\in X: \Sigma_p(X) \text{ has boundary}\}.
\end{equation}
When $N=1$ Alexandrov spaces are manifolds, the boundary is just boundary of a manifold, see \cite[7.19]{BGP92}. This gives the inspiration to the Kapovitch-Mondino boundary \eqref{eq:KMboundary}. It is clear that when $(X,\dist,\haus^N)$ is viewed as a $\ncRCD(K,N)$ space, this boundary is exactly the Kapovitch-Mondino boundary, which justifies the use of notation.

Similar to the $\ncRCD$ case, the set of manifold points of $X$ is open and dense, the non manifold points of $X$ is of Hausdorff dimension and topological dimension at most $n-1$ if $X$ has boundary and codimension at most $n-2$ if $X$ does not have boundary. This follows  by combining \eqref{eq:sing}, \eqref{eq:boundaryrelation} and Theorem \ref{thm:regular}.  

We will also need the notion and properties of quasigeodesics on Alexandrov spaces \cite{PP-quasigeoodesics}. Recall that a unit speed curve $\gamma$ in an Alexandrov space is called a \emph{quasigeodesic} if restrictions of distance functions to $\gamma$ have the same concavity properties as their restrictions to geodesics. For example, for non-negatively curved Alexandrov space $X$ this means that for any $p\in X$ the function $t\mapsto d(\gamma(t), p)^2$ is 2-concave. Every geodesic is obviously a quasigeodesic but the converse need not be true. For example if $X$ is the unit disk in $\R^2$ then the boundary circle is a quasigeodesic in $X$.
Petrunin and Perelman showed  \cite{PP-quasigeoodesics} that for every point $p$ in an Alexandrov space there are infinite quasigeodesic starting in every direction at $p$.

\section{Invariance of Domain for Alexandrov spaces}\label{sec:inv}

As stated in the introduction, the invariance of domain for Alexandrov spaces has long been known by experts, we present here a precise statement and its proof due to Belegradek-Ivanov-Pertunin on mathoverflow \cite{BIP10}. 

\begin{theorem}\label{thm:inv}
Let $(X,\dist_X)$, $(Y,\dist_Y)$ be Alexandrov spaces of same dimension, $f:X\to Y$ be a injective continuous map. For any open subset $U\subset X$, if $U\cap \mathcal{F}X=\varnothing$  then $f(U)\cap \mathcal{F}Y=\varnothing$, and $f(U)$ is open in $Y$. 
\end{theorem}

This theorem follows from the following purely topological Invariance of Domain Theorem for MCS spaces.

\begin{theorem}\label{thm:inv-dom-mcs}
{
Let $X,Y$ be $n$ dimensional MCS spaces such that $X_{n-1}=Y_{n-1}=\varnothing$ and for all points in $Y$ their open conical neighborhoods have connected $n$-strata.
}

Let $f: X\to Y$ be continuous and injective.

Then $f(X)$ is open in $Y$ {and open conical neighborhoods of all points in $X$ have connected top strata.}

\end{theorem}

We need the following lemma regarding the  $\mathbb{Z}_2$-cohomology for 
{ MCS spaces originated from Grove-Petersen \cite{PG93}, initially stated for compact Alexandrov spaces without boundary}. Note that finite dimensional 
MCS spaces are locally compact, and locally contractible, since every point has a neighborhood homeomorphic to a cone, so Alexander-Spanier cohomology, singular cohomology and Cech cohomology all coincides. It is not necessary to specify which cohomology to use. In what follows all cohomology is taken with $\Z_2$ coefficients.

We will make use of the following duality which holds for Alexander-Spanier cohomology with compact support \cite[Chapter 1]{Massey-book}.
Given a locally compact and Hausdorff space $Y$ and a closed subset $A\subset Y$ it holds that $H^n_c(Y,A)\cong H^n_c(Y\setminus A)$.
{ 
\begin{lemma}\label{lem:GP}
Let $X$ be an $n$-dimensional compact MCS space where $X_n$ has $k$ connected components and $X_{n-1}=\varnothing$, then $H^n(X)\cong\mathbb{Z}_2^k$.
\end{lemma}
}
\begin{proof}
The proof is the same as in \cite{PG93}.

Since $X_n=X\setminus S$ is an $n$-manifold with $k$ connected components we have that $H^n_c(X\setminus S)\cong \Z_2^k$. On the other hand by Alexander-Spanier duality we have that $H^n_c(X\setminus S)\cong H^n_c(X,S)\cong H^n(X,S)$ where the last isomorphism holds since $X$ is compact. Now the result immediately follows from the long exact sequence of the pair $(X,S)$ using the fact that $S$ is the union of strata of dimension $\le 2$ and hence $H^{n-1}(S)\cong H^{n}(S)=0$.
\end{proof}

 Note that in the above proof we get that $H^n_c(X\setminus S)\cong H^n(X,S)\cong H^n(X)$.
Compare this to the proof of the following Lemma

\begin{lemma}\label{lem:Igor}
    Let $(X,\dist_X)$ be a compact $n$-dimensional MCS space with connected $X_n$ and ${X_{n-1}}=\varnothing$, take $x\in X_n$. 
 Then we have
    \begin{enumerate}
     \item\label{lem:Igoritem1} $H^n(X\setminus \{x\})=0$;
     \item\label{lem:Igoritem2} the inclusion $i: (X,\varnothing)\to (X,X\setminus\{x\})$ induces an isomorphism on cohomology, that is  
        \begin{equation}
         i^*: H^n(X, X\setminus \{x\})\to H^n(X)
        \end{equation}
    is an isomorphism.
    \end{enumerate}
\end{lemma}

\begin{proof}[Proof of Lemma~\ref{lem:Igor}]
{

 We first show item \ref{lem:Igoritem2}. Let $U\subset X\setminus S=X_n$ be  connected and open. Since $X\setminus S$ is a manifold and $U$ is connected, we have that the inclusion $U\hookrightarrow X\setminus S$ induces an isomorphism between  compactly supported cohomology $H^n_c(U)$ and $H^n_c (X\setminus S)$.

 Also, since $X$ is compact we have that $H^n_c(X, X\setminus U)\cong H^n(X, X\setminus U)$ and similarly $H^n_c(X,S)\cong H^n(X,S)$.

  With this at disposal, consider the inclusion of pairs $(X,S)\hookrightarrow (X,X\setminus U)$, we have 
 
 \begin{tikzcd}
H^n_c(U) \arrow[r, "\cong"] \arrow[d,"\cong"]
& H^n_c(X\setminus S) \arrow[d, "\cong"] \\
H^n(X, X\setminus U) \arrow[r]
& |[]| H^n(X,S),
\end{tikzcd}\\

where the vertical arrows are Alexander-Spanier duality combined with the above isomorphisms  $H^n_c(X, X\setminus U)\cong H^n(X, X\setminus U)$ and $H^n_c(X,S)\cong H^n(X,S)$.
}

This gives an isomorphism between $H^n(X, X\setminus U)$ and $H^n(X,S)$ hence between $H^n(X, X\setminus U)$ and $H^n(X)$ by inclusion. Note that $X\setminus \{x\}$ deformation retract to $X\setminus U$ for some open conical open neighborhood $U\subset X\setminus S$ of $x$, which implies that $i^*: H^n(X, X\setminus \{x\})\to H^n(X)$ is an isomorphism. 

Next we show item \ref{lem:Igoritem1}. To compute $H^n(X\setminus \{x\})$, look at the long exact sequence for the pair $(X,X\setminus \{x\})$:
 \begin{equation}
     \cdots\rightarrow H^n(X,X\setminus \{x\})\xrightarrow{\cong}  H^n(X)\rightarrow H^n(X\setminus \{x\})\xrightarrow{0} H^{n+1}(X,X\setminus \{x\})\rightarrow\cdots,
 \end{equation}
$H^n(X\setminus \{x\})=0$ follows directly. 
\end{proof}

Now we can prove the invariance of domain for MCS spaces. The strategy is to localize $X,Y$ to suspensions over lower dimensional strata 
, so that the proof reduces to the case of compact MCS 
spaces with connected top stratum and empty codimension 1 stratum, where the above lemmas apply.

\begin{proof}[Proof of Theorem \ref{thm:inv-dom-mcs}]

{ Let us first prove the theorem under the extra assumption that for all points in $X$ the top strata of their conical neighborhoods are connected.}

We break the proof into steps.

{\bf Step 1: }Localize to suspensions, which are MCS spaces satisfying assumption in Lemma \ref{lem:GP} and Lemma \ref{lem:Igor}.

Let $x\in U$ and $y\defeq f(x)\in f(U)$. Both $x,y$ have a neighborhood homeomorphic to  cones over 
{ some $(n-1)$-dimensional MCS spaces} $\Sigma_x$, $\Sigma_y$, respectively. Take cone neighborhoods of $x$, $B_x\Subset B'_x\Subset U$, then there exists a cone neighborhood of $y$, say $B_y\subset f(U)$ such that $B_y\cap f(\overline{B'_x}\setminus B_x)=\varnothing$.  Let $C\defeq U\setminus B_x$ and $D\defeq Y\setminus B_y$. Note that both $U/C$ and $Y/D$ are homeomorphic to a suspension over $\Sigma_x$, $\Sigma_y$ respectively. The quotient map induces a new map $\tilde{f}: U/C\to Y/D$ between compact {$n$-dimensional MCS spaces with connected top stratum and empty codimension $1$ stratum. }
Observe that $\tilde{f}$ remains injective on $f^{-1}(B_y)=f^{-1}(Y\setminus D)$. 

It suffices to show that $\tilde{f}$ is surjective onto $B_y$ identified with its image in $Y/D$. By continuity, it suffices to show every 
point in $Y_n \cap B_y$ is in the image of $\tilde{f}$.

{\bf Step 2: }We show that $\tilde{f}^*: H^n(Y/D)\to H^n(U/C)$ is an isomorphism.

 First, we claim that there exists a 
 point $x'\in X_n $ such that $y'\defeq f(x)\in Y_n$. To see this, let $x\in U\cap X_n$, and take a compact neighborhood $B$, it is of topological dimension $n$, since $f$ is injective and continuous, it is a homeomorphism between $B$ and $f(B)$, so $f(B)$ also has topological dimension $n$, which means $f(B)$ can not be entirely in $\cup_{k=0}^{n-2}Y_k$, 
 which is of topological dimension at most $n-2$. Now that we have $x'\in X_n$ and $y'\in Y_n$ , we claim that $\tilde{f}^*:H^n(Y/D, Y/(D\setminus \{y'\}))\to H^n(U/C,U/(C\setminus\{x'\}))$ is an isomorphism. 

To this end, take an excision around the manifold neighborhood of $x',y'$ respectively. The desired claim reduces to showing that $f^*:H^n(B^n,B^n\setminus \{x'\})\to H^n(f(B^n),f(B^n)\setminus \{y'\})$ is an isomorphism for injective and continuous $f$ such that $f(x')=y'$, where $B^n$ is a ball in $\setR^n$. The invariance of domain for $\setR^n$ has been used to show that $f(B^n)$ is open so that an excision can be applied on $Y/D$. The invariance of domain for $\setR^n$ also shows $f:(B^n, B^n\setminus \{x'\})\to (f(B^n),f(B^n)\setminus \{y'\})$ is a homeomorphism, the claim follows.

Now consider the induced map $f^*$ between long exact sequences of the pairs $(Y/D, Y/(D\setminus \{y'\}))$ and $(U/C,U/(C\setminus \{x'\}))$, taking also into account item \ref{lem:Igoritem2} of Lemma \ref{lem:Igor}, by 5-Lemma it follows that $\tilde{f}^*: H^n(Y/D)\to H^n(U/C)$ is an isomorphism.

{\bf Step 3: }Arguing by contradiction assume that 
 $\tilde{f}$ is not surjective onto $(Y_n\cap B_y)$ identified with its image in $Y/D$, we show that $\tilde{f}^* :H^n(Y/D)\to H^n(X/C)$ is a zero map. However, $\tilde{f}^*$ cannot be both zero map and isomorphism (from Step 2), because by Lemma \ref{lem:GP}, $H^n(Y/D)= H^n(U/C)=\mathbb{Z}_2$, a contradiction.

For this purpose, suppose that a 
point $z\in Y_n\cap B_y$ is missed by $\tilde{f}$, then $\tilde{f}$ can be factored through 
\begin{equation}
    \tilde{f}: U/C\rightarrow Y/(D\setminus \{z\})\rightarrow Y/D.
\end{equation}
Since $H^n(Y/(D\setminus \{z\}))=0$ due to item \ref{lem:Igoritem1} of lemma \ref{lem:Igor}, $\tilde{f}^* :H^n(Y/D)\to H^n(X/C)$ is a zero map. 

{
This concludes the proof of the theorem under the extra assumption that for all points in $X$ the top strata of their conical neighborhoods are connected.

To complete the proof in the general case we will need the following general lemma.

\begin{lemma}\label{lem-top-strata}
Let $Z$ be a connected  $n$-dimensional MCS space space that that it's top stratum $Z_n$ is not connected. Then there exists a point $z\in Z$ such that the top stratum  of its conical neighborhood $U_z$ is not connected.
\end{lemma}
\begin{proof}[Proof of Lemma \ref{lem-top-strata}]
let $p, q$ be points lying in different connected components of $Z_n$. Since $Z$ is connected there is a path $\gamma:[ 0,1]\to Z$ such that $\gamma(0)=p, \gamma(1)=q$.
By compactness of $[0,1]$ there exists finitely many connected components $U_1,\ldots U_k$  of $Z_n$ whose closures intersect $\gamma$. Since the top stratum  is dense in $Z$ we have that
$\gamma$ is contained in $\bar U_1\cup \bar U_2\cup\ldots\cup \bar U_k$, therefore $[0,1]=\gamma^{-1}(\bar U_1)\cup \gamma^{-1}(\bar U_2)\cup \ldots\cup \gamma^{-1}(\bar U_k)$.
As all these sets are closed and $[0,1]$ is connected this covering can not be disjoint and hence  there is $t_0\in [0,1]$ which belongs to at least two $\gamma^{-1}(\bar U_j)$. Then $z=\gamma(t_0)$ satisfies the conclusion of the Lemma.

\end{proof}
We now continue with the proof of Theorem \ref{thm:inv-dom-mcs}.

Recall that we have proved the theorem under the assumption that all conical neighborhood of points in $X$ have connected top strata.

Now suppose there are some points in $X$ such that the top strata of their conical neighborhoods are not connected. Let $l$ be the largest number that $X_l$ contains such a point $x$. Take such $x\in X_l$.

Then its conical neighborhood $U_x$ has the form $\R^l\times C(\Sigma)$ where $\Sigma$ is $(n-l-1)$-dimensional MCS space. 
Note that points in $U_x$ outside of $\R^l\times \{*\}$ (here $*$ is the cone point in $C(\Sigma)$)  lie in the union of strata of dimension $>l$.

We claim that $\Sigma$ has more than one connected components. Indeed, if not then its top stratum is not connected while $\Sigma$ itself is connected. Then by Lemma \ref{lem-top-strata} applied to $\Sigma$  there exists a point $\sigma\in \Sigma$ such that the top stratum  of its conical neighborhood in $\Sigma$ is not connected. But then the corresponding point in $U_x$ will lie in $X_m$ for $m>l$ and also have the property that its conical neighborhood has more than one top stratum components. This contradicts the maximality of $l$ in the choice of $x$. 

Let $\Sigma'$ be one component of $\Sigma$. Then the subset $W'=\R^l\times C(\Sigma')\subset U_x$ is an $n$-dimensional MCS space with empty $(n-1)$-stratum and such that the top stratum of all conical neighborhoods in $W'$ are connected.
 Then we have an injective embedding $f: W'\to Y$ and by the proof above the image $f(W')$ is an open neighborhood of $f(x)$. But the same argument applies to any other component $\Sigma''$ of $\Sigma$ and gives another subset $W''\subset U_x$ which contains $x$ and such that $f(W'')$ is also an open neighborhood of $f(x)$. This contradicts injectivity of $f$ near $x$. Therefore under the assumption of the theorem conical neighborhoods of points in $X$ must necessarily have connected top strata. }
\end{proof}

\begin{remark}
The connectedness assumption of top strata of conical neighborhoods in $Y$ is essential. For example, take $Y=\R^n\bigvee  \R^n$ to be the wedge sum of two copies of $\R^n$ glued at $0$, $X=\R^n$ and $f:X\hookrightarrow Y$ be inclusion of the first copy of $\R^n$. 
This map is clearly  1-1 but the image is not open since it does not contain any neighborhood of $0$ in $Y$.

\end{remark}

\begin{remark}
The conclusion that conical neighborhoods of points in $X$ must be connected can be viewed as a non-embeddability result. In other words the following holds. Suppose $Y$ satisfies the assumption of the theorem and $X$ is an $n$-dimensional MCS space with empty $(n-1)$-stratum and such that there is a point in $X$ such that the top stratum of its conical neighborhood is not connected. Then there is no 1-1 continuous map $f:X\to Y$.
\end{remark}

\begin{proof}[Proof of Theorem \ref{thm:inv}]
	As pointed out in section \ref{subsec:Alexandrov}, every Alexandrov space is an MCS space with connected top stratum.
	{
	 The assumption $U\cap \mathcal{F}X=\varnothing$ implies that $U$ has empty codimension $1$ stratum. } Next, for every $p\in Y$ its conical neighborhood $W_p$ is homeomorphic to $T_pY$ which is a nonnegatively curved Alexandrov space.
	
 Since the top stratum  of $T_pY$ is connected the same is true for $W_p$.
	
	 It suffices to show that $f(U)\cap\mathcal{F}Y=\varnothing$, everything else follows from Theorem \ref{thm:inv-dom-mcs}. 
	
	Assume $\mathcal{F}Y\neq \varnothing$. Take the metric double $\tilde{Y}$ of $Y$, $\tilde{Y}$ an $n$-dimensional Alexandrov space without boundary, and $f:X\to Y$ extends to an injective and continuous map into $\tilde Y$ by post composing with the inclusion map $Y\hookrightarrow \tilde Y$. We still denote it by $f$. Applying Theorem \ref{thm:inv-dom-mcs} to $f:X\to \tilde{Y}$, we see that $f(U)$ must be open in $\tilde Y$. If there exists $z\in f(U)\cap \mathcal{F}Y$, then there exists an open neighborhood $V$ of $z$ in $f(U)\cap\tilde Y$. By definition of metric double $V$ must intersect both copies of $Y$ in $\tilde Y$, this is a contradiction to the definition of $f$, from which it follows that $f(U)$ can not intersect $\mathcal{F}Y$.
\end{proof}



\section{Equivalence of intrinsic and extrinsic boundary}\label{sec:equi}

\subsection{Alexandrov case}

\begin{proof}[Proof of Theorem \ref{thm:main1}]

    We first show that $\partial_{\rm top}\bar\Omega=\mathcal{F}\bar\Omega$. Since tangent cones at points in $\Omega$ have no boundary, we see that $\mathcal{F}\bar\Omega\subset\partial_{\rm top}\Omega$. Now take $p\in \partial_{\rm top}\bar\Omega$, if to the contrary $p\notin \mathcal{F}\bar\Omega$, then there is an open set $U\subset \bar{\Omega}$ containing $p$ such that $U\cap \mathcal{F}\bar\Omega=\varnothing$, since $\mathcal{F}\bar\Omega$ is closed. The Invariance of Domain Theorem \ref{thm:inv} applied to inclusion $i:\bar{\Omega}\hookrightarrow X$ yields that $i(U)=U$ is also an open subset of $X$, so $p\in U\subset {\rm Int}_{\rm top}(\bar\Omega)=\Omega$, a contradiction to $p\in\partial_{\rm top}\bar\Omega$.  

    Now that $\partial_{\rm top}\bar\Omega=\mathcal{F}\bar\Omega$, it follows immediately that $\Omega$ coincides with $\bar\Omega\setminus \mathcal{F}\bar\Omega$, which is the interior in the sense of Alexandrov spaces. So strong convexity of the interior of an Alexandrov space yields that $\gamma$ does not intersect $\mathcal{F}\Omega$ hence $\partial_{\rm top}\Omega$. 
    The proof of \ref{main1:item2} is completed by noticing that any $\dist_\Omega$ geodesic connecting points in ${\rm Int}_{\rm top}(\bar\Omega)=\Omega$ and entirely contained in $\Omega$ is a local geodesic of $(X,\dist_X)$.

For the proof of item \ref{main1:item3}, let $p,q\in \partial{\rm Int}_{\rm top}(\bar\Omega), d=\dist_\Omega(p,q)$, and $\gamma: [0,d]\to \bar\Omega$ be a unit  speed geodesic with respect to $\dist_\Omega$  joining $p,q$ such that $\gamma(0)=p$, $\gamma(1)=q$. For any small enough $\eps\in (0,d/3)$, take $p'=\gamma(\eps)$ and $q'=\gamma(d-\eps)$. 
{
We can find points in  $\{p'_n\}$ and $\{q'_n\}$  in $\Omega$ so that $p'_n\to p'$ and $q'_n\to q'$. The geodesic $\gamma_n$ of $(\bar\Omega,\dist_\Omega)$ joining $p_n$ and $q_n$  must  converge to $\gamma|_{[\eps,1-\eps]}$ otherwise there would have been branching geodesics between $p,q$. 
On the other hand $\gamma_n$ is a local geodesic of $(X,\dist_X)$ and hence is a quasigeodesic in $X$. Since limits of quasigeodesics are quasigeodesics it follows that 
 $\gamma|_{[\eps,1-\eps]}$   is a quasi-geodesic in $X$. Letting $\eps\to 0$ we conclude that $\gamma$ is a quasi-geodesic in $X$ as well.}
\end{proof}

\subsection{$\ncRCD$ case}
The purpose of this section is to prove Theorem \ref{thm:main2}. We need the following pairwise almost convexity proved by Deng in \cite[Theorem 6.5]{deng2020holder}.

\begin{proposition}\label{prop:pairconv}
Let $(X,\dist,\meas)$ be an $\RCD(K,N)$ space with ${\rm essdim}=n$. For $\meas\times\meas$-a.e.\ every $(x,y)\in\mathcal{R}_n\times \mathcal{R}_n$, there exists a geodesic joining $x,y$, and entirely contained in $\mathcal{R}_n$.
\end{proposition}

 \begin{proof}[Proof of Theorem \ref{thm:main2}]
 We assume that $\bar\Omega\neq X$, otherwise it either contradicts the assumption $\bar \Omega\cap \partial X=\varnothing$ or makes the statement trivial. We break the proof into several steps.
 
 {\bf Step 1:} we show $\partial \bar\Omega\subset \partial_{\rm top} \bar\Omega$.

 Observe that $\dist_\Omega$ and $\dist_X$ coincide on sufficiently small open subsets of $\Omega$, hence tangent cones taken at the same point by the same rescaling sequence w.r.t. both metrics are isometric for points in $\Omega$, in particular tangent cones at points in $\Omega$ have no boundary. Which means that  $\mathcal{S}^{N-1}(\bar\Omega)\setminus \mathcal{S}^{N-2}(\bar\Omega)\subset \partial_{\rm top} \bar\Omega$. Since $\partial_{\rm top} \bar\Omega$ is closed, we have $\partial \bar\Omega\subset \partial_{\rm top} \bar\Omega$.
 
 {\bf Step 2:} Suppose $\partial_{\rm top} \bar \Omega\subset \partial \bar\Omega$ is not true, we find a point $q\in \partial_{top}\bar\Omega\setminus \partial \bar\Omega$ so that $q\in \mathcal{R}(X)$.

 First, there exists $p\in \partial_{\rm top}\bar\Omega\setminus \partial \bar\Omega$. Since $\partial_{\rm top}\bar\Omega$ and $\partial \bar\Omega$ are both closed, there exists $\eps>0$ such that $B_{2\eps}(p)\cap \partial \bar\Omega=\varnothing$. Now consider any two points in $B_{\eps/2}(p)$. By triangle inequality, any geodesic joining such two points lies in $B_{\eps}(p)$ hence does not intersect $\partial \bar\Omega$, moreover, note that $\haus^N_X(B_{\eps/2}(p)\cap \Omega)>0$, $\haus^N_X(B_{\eps/2}(p)\cap (X\setminus \bar\Omega))>0$ (recall we assumed $\bar\Omega\neq X$), by Deng's pairwise almost convexity of the regular set, Proposition \ref{prop:pairconv}, there exist $x\in B_{\eps/2}(p)\cap \Omega\cap \mathcal{R}(X)$ and $y\in B_{\eps/2}(p)\cap (X\setminus \bar\Omega)\cap \mathcal{R}(X)$ such that some geodesic, denote it by $\gamma_{xy}$, joining $x,y$ is entirely contained in $\mathcal{R}(X)$, meanwhile, $\gamma_{xy}$ must intersect $\partial_{\rm top}\bar\Omega$, and the point of intersection, denoted by $q$, is the desired point. 
 
 {\bf Step 3:} We show that for the point  $q$ we found in step $2$, there exists a neighborhood $U$ so that $\partial_{\rm top}\bar \Omega\cap U$ has Hausdorff codimension at least $2$ (recall Remark \ref{rmk:equiHaus}), and there exists $\delta\defeq \delta(K,N)>0$ depending only on $K,N$ such that $\Theta_{\bar\Omega}(x)\le 1-\delta$ for any $x\in \partial_{\rm top}\bar \Omega\cap U$.

 Since $q\in \mathcal{R}(X)\cap (\partial_{\rm top}\bar \Omega\setminus \partial \bar\Omega)$, there exists an open neighborhood $U$ such that $U$ is homeomorphic to a manifold and $U\cap \partial \bar\Omega=\varnothing$.  We claim that $\partial_{\rm top}\bar \Omega\cap U\subset \mathcal{S}^{N-2}(\bar\Omega)$. It suffices to show $\partial_{\rm top}\bar \Omega\cap U\subset \mathcal{S}(\bar\Omega)$ since $U$ is disjoint from $\partial \bar\Omega$. 
 
 Let $\delta\defeq\delta(K,N)>0$ be as in Theorem \ref{thm:regular}, if there exists $x\in \partial_{\rm top}\bar \Omega\cap U$ with $\Theta_{\bar\Omega}(x)> 1-\delta$ then there exists $V\subset U\cap \bar\Omega$ containing $x$, open relative to $\bar\Omega$, and homeomorphic to a manifold. Now the invariance of domain for manifolds applied to the inclusion $V\hookrightarrow U$ yields that $V$ is open in $X$, hence $V\subset \Omega$. This contradicts that $x\in \partial_{\rm top}\bar \Omega$. Therefore  for any $x\in \partial_{\rm top}\bar \Omega\cap U$ it holds that $\Theta_{\bar\Omega}(x)\le  1-\delta$ which by the choice of $\delta$ implies that  $\partial_{\rm top}\bar \Omega\cap U\subset \mathcal{S}^{N-2}(\bar\Omega)$. 
  Since Hausdorff codimension of $\mathcal{S}^{N-2}(\bar\Omega)$ is at least $2$, the proof of this step is completed.

 {\bf Step 4:} We show that when we blow up the inclusion map $ i_0:\bar\Omega\hookrightarrow X$ at $q$, the induced map $i_1: T_q\bar \Omega\to T_q X  \cong \R^N$ is not surjective near $0$, in fact, $0$ is on the topological boundary of $i_1(T_q\bar \Omega)$.
 
 Denote by $B^X_r$ (resp. $B^{\bar\Omega}_r$) the ball of radius $r$ in metric $\dist_X$ (resp. $\dist_\Omega$). We claim that $\haus^N_X(B^{\bar\Omega}_r(x))=\haus^N_{\bar\Omega}(B^{\bar\Omega}_r(x))$ for $x\in U\cap \bar\Omega$ and $r>0$ small enough so that $B^{\bar\Omega}_r(x)\subset U$. Observe that the two distances $\dist_X$ and $\dist_{\Omega}$ coincide with each other for small enough open subsets in $\Omega$, so $\haus^N_X$ and $\haus^N_{\bar\Omega}$ gives the same mass to open subsets of $\Omega$. Now observe that $B^{\bar\Omega}_r(x)=(B^{\bar\Omega}_r(x)\cap \Omega)\cup (B^{\bar\Omega}_r(x)\cap\partial_{\rm top} \bar\Omega)$, where the former is open in $\Omega$, the latter has codimension at least $2$ proved in step 3 hence measure zero, which completes the proof of the claim. Recall from step 2 and step 3 we know that $\Theta_X(p)=1$ and $\Theta_{\bar\Omega}(p)\le 1-\delta$, it follows
 
 \begin{equation}\label{eq:density}
     \lim_{r\to 0} \frac{\haus^N_X(B^{\bar\Omega}_r(p))}{\haus^N_X(B^{X}_r(p))}=\lim_{r\to 0} \frac{\haus^N_{\bar\Omega}(B^{\bar\Omega}_r(p))}{\haus^N_X(B^{X}_r(p))}=\frac{\Theta_{\bar\Omega}(p)}{\Theta_{X}(p)}\le 1-\delta.
 \end{equation} 
 If $i_1(T_q \bar \Omega)$ contains $B^{\R^N}_{\eps}(0)$ for some $\eps>0$, then the local coincidence of the metrics when away from boundary implies $B^{\R^N}_{\eps/2}(0)=B^{T_q\bar\Omega}_{\eps/2}(0)$, which in turn implies $\haus^N_{\R^N}(B^{T_q\bar\Omega}_{\eps/2}(0))=\haus^N_{\R^N}(B^{\R^N}_{\eps/2}(0))$, this contradicts \eqref{eq:density}.

 {\bf Step 5:} We derive a contradiction by iteratively blowing up at a topological boundary point.

 If $N=1$, then the statement is clear thanks to the classification theorem \cite{KL15}. It suffices to consider the case $N\ge 2$. In this case the topological boundary of $T_q\bar\Omega$ is more than a single point, to show this, it is enough to notice that $i_1$ is bi-Lipschitz (recall remark \ref{rmk:equiHaus}), so it is an homeomorphism onto its image. 
 
 Now we summarize the properties needed for the blow-up procedure. In the setting of this theorem, let $i_0:\bar \Omega\hookrightarrow X$ be the inclusion map, $q\in \partial_{\rm top}\bar\Omega$ and $i_1: T_q\bar\Omega\to T_qX$ be the blow-up of $i_0$ at $q$. In order for the cone tip of $T_q\bar\Omega$ to be on $\partial_{\rm top}i_1(T_q\bar\Omega)$, it is sufficient to have:
 
 \begin{enumerate}
     \item $q\in \mathcal{R}(X)$ and $\Theta_{\bar\Omega}(q)\le 1-\delta$; 
     \item $\haus^N_{\bar\Omega}(B^{\bar\Omega}_r(q))=\haus^N_{X}(B^{\bar\Omega}_r(q))$ for sufficiently small $r>0$;
     \item $q\notin \mathcal{F}\bar\Omega$.
 \end{enumerate}
 
 After the blow-up procedure in step 4, the ambient space $T_q X\cong \R^N$ has no singular points, moreover, $q\notin \partial \bar\Omega$ implies $q\notin  \mathcal{F}\bar\Omega$, which means iterated tangent cones at $q$ w.r.t. $(\bar\Omega, \dist_\Omega)$ have no boundary, so every point on  $\partial_{\rm top} i_1(T_q\bar\Omega)$ (not empty by step 4) still satisfies the conditions listed above, so we can continue blowing up at any point on $\partial_{\rm top} i_1(T_q\bar\Omega)$ other than the cone tip, each time keeping the the base point a point on the topological boundary.  In finitely many blow-up procedures, we end up with a bi-Lipschitz map $i_N: \R^N\to \R^N$ such that $i_N(0)=0$, $i_N$ not surjective, and $0$ is on the topological boundary of $i_N(\R^N)$, this is impossible by invariance of domain.
 
 \end{proof}

\section{Applications}\label{sec:AppConv}
 In this section we derive from the boundary equivalence in various ambient spaces the {locally totally geodesic} property, i.e., a subset satisfying $\ncRCD(K,N)$ condition forces the geodesics in intrinsic metric joining interior points to be disjoint from boundary. 
 
 We first introduce a technical result which is a direct consequence of H\"older continuity along interior of tangent cones pointed out in \cite[Corollary 1.5]{CN12}, it is available for $\ncRCD(K,N)$ spaces thanks to Deng's generalization of this statement \cite{deng2020holder}. 

\begin{proposition}\label{prop:dentofull}
Let $(X,\dist,\meas)$ be an $\RCD(K,N)$ space, and $\gamma$ be a geodesic in $X$. The set of points in $\gamma$ with unique tangent cone is relatively closed in the interior of $\gamma$. In particular, for each integer $1\le k\le N$, $\gamma\cap \mathcal{R}_k$ is closed relative to the interior of $\gamma$. If in addition $\gamma\cap \mathcal{R}_k$ is dense in the interior of $\gamma$, then it is all of the interior.
\end{proposition}

We start with the following simplest setting, where the ambient space is a smooth manifold {but there are no assumption on the regularity of topological boundary}.
\begin{theorem}\label{thm:smooth-rcd-subset}
    Let $(M,g)$ be an $n$-dimensional smooth
    manifold, and $\Omega\subset M$ be open, connected and such that ${\rm Int}(\bar\Omega)=\Omega$. If $(\bar\Omega,\dist_\Omega, \vol_g\mres \bar\Omega)$ is a $\ncRCD(K,n)$ space, then 
    \begin{enumerate}
        \item $\partial_{\rm top}\bar \Omega=\partial \bar\Omega$.
        \item any minimizing geodesic in $(\bar\Omega,\dist_\Omega)$ joining two points in $\Omega$ does not intersect $\partial \Omega$ hence a local geodesic in $(M,g)$, {i.e., $\Omega$ is locally totally geodesic};
        \item any minimizing geodesic in $(\bar\Omega,\dist_\Omega)$  joining two points on $\partial_{\rm top}\bar\Omega$ is either entirely contained in $\partial_{\rm top}\bar\Omega$, or its interior is entirely in $\Omega$. In the latter case the minimizing geodesic is also a local geodesic in $(M,g)$.
        
    \end{enumerate}
\end{theorem}

\begin{proof}
   We first show that if $p\in \partial_{\rm top}\bar\Omega$, then any tangent cone taken w.r.t. $(\bar\Omega,\dist_\Omega, \vol_g\mres \bar\Omega)$ at $p$ cannot be $\R^n$. This is contained in step 3 of the proof of Theorem \ref{thm:main2}. If there is a tangent cone w.r.t. $(\bar\Omega,\dist_\Omega, \vol_g\mres \bar\Omega)$ at $p$ is $\R^n$, then there is a neighborhood $V$ of $p$ open in $\bar\Omega$ homeomorphic to $\R^n$, while there is also a neighborhood $U$ of $p$ open in $M$ homeomorphic to $\R^n$. Then the invariance of domain applied to the inclusion $U\cap V\hookrightarrow U$ show that $U\cap V$ is open in $M$ and $U\cap V\subset \Omega$, a contradiction to $p\in \partial_{\rm top}\bar\Omega$. It follows directly that $\partial_{\rm top}\bar\Omega=\partial \bar \Omega$.
   
    Consider now a minimizing geodesic $\gamma:[0,1]\to \bar\Omega$ in $(\bar\Omega,\dist_{\Omega})$. Then $\gamma((0,1))\cap \partial_{\rm top} \bar\Omega$ is relatively closed in $\gamma((0,1))$. By Proposition \ref{prop:dentofull}, $\gamma((0,1))\setminus \partial_{\rm top} \bar\Omega$ is also relatively closed, this is the set of points in $\gamma((0,1))$ having tangent cone $\R^n$. It follows from the connectedness of $\gamma((0,1))$ that either $\gamma((0,1))\setminus \partial_{\rm top} \bar\Omega$ or $\gamma((0,1))\cap \partial_{\rm top} \bar\Omega$ is empty.
\end{proof}

{
\begin{remark}
Note that Theorem \ref{thm:smooth-rcd-subset} implies that $\bar \Omega$ is locally convex in $M$ and is hence locally Alexandrov (globally Alexandrov if it is compact).
\end{remark}
}
We now move to the case where the ambient space is a $\ncRCD(K,N)$ space. With the extra assumption $\partial X=\mathcal{F}X$ and the stability of absence of boundary \cite[Theorem 1.6]{BNS20} of an $\RCD(K,N)$ space, the exact same idea can be used to prove that ${\rm Int}(X)$ is strongly geodesically convex.

\begin{corollary}\label{cor:intconv}
 Let $(X,\dist,\haus^N)$ be a $\ncRCD(K,N)$ space. Assume $\partial X=\mathcal{F} X$, then ${\rm Int}(X)\defeq X\setminus \partial X$ is strongly convex, i.e.\ any geodesic joining points in ${\rm Int}(X)$ does not intersect $\partial X$.
\end{corollary}

\begin{proof}
    For a constant speed geodesic $\gamma:[0,1]\to X$ joining two points in ${\rm Int}(X)$, if $\gamma\cap \partial X\neq \varnothing$, then there exists a $t_0\in (0,1)$ such that $t_0=\sup\{t: \gamma([0,t))\cap \partial X\}=\varnothing$ and $\gamma (t_0)\in \partial X$, since $\partial X$ is closed. Note that $\gamma(t_0)$ is an interior point of $\gamma$ and for every $t\in (0,t_0)$, any tangent cone at $\gamma(t)$ does not have boundary, now the he stability of absence of boundary \cite[Theorem 1.6]{BNS20} under pmGH convergence and h\"older continuity of tangent cones along the interior of a geodesic yield that any tangent cone at $\gamma(t_0)$ has no boundary, this contradicts $\gamma(t_0)\in  \partial X=\mathcal{F}X$.
\end{proof}

	\begin{corollary}\label{cor:loc-total-geo}
		In the setting of Theorem \ref{thm:main2}, with the extra assumption that $\mathcal{F}X=\partial X$, any (minimizing) geodesic in $(\bar\Omega,\dist_\Omega)$ joining two points in $\Omega$ is a local geodesic in $(X,\dist_X)$, hence $ \Omega$ is locally totally geodesic.
	\end{corollary}


If we consider only a noncollapsed Ricci limit space with boundary $(X,\dist,\meas)$, i.e. the pmGH limit of $n$-dimensional manifolds with convex boundary and uniform Ricci curvature lower bound in the interior and uniform volume lower bound of ball of radius $1$ centered at points chosen in the pmGH convergence, then $\mathcal{F}X=\partial X$ is already verified \cite[Theorem 7.8]{BNS20}, naturally we have:

\begin{corollary}
Let $(X,\dist,\meas)$ be a noncollapsed Ricci limit space with boundary, then its interior $X\setminus\partial X$ is strongly convex.
\end{corollary}

Due to the lack of a notion of intrinsic boundary for collapsed spaces we have been discussing noncollapsed spaces only. Without stratification of singular set, De Philippis-Gigli definition's cannot be applied, and Kapovitch-Mondino's definition also fails to provide the correct definition of boundary, as the metric horn example by Cheeger-Colding \cite[Example 8.77]{Cheeger-Colding97I} shows a collapsed Ricci limit space can have an interior cusp at which the tangent cone is a half line. Nevertheless we conjecture that Han's Theorem \ref{thm:han} holds in much larger generality, that is, a subspace in a $\ncRCD(K,N)$ ambient space along with some reference measure satisfying $\RCD(K,\infty)$ condition should still enjoy the property that geodesics in the intrinsic metric joining points in the interior remains away from boundary, and the reference measure gives measure $0$ to the topological boundary. This would provide a partial converse ( different from local-to-global theorem) to the well-known global-to-local theorem for $\RCD(K,\infty)$ spaces from \cite[Theorem 6.18]{AGS14a}:  

 \begin{theorem}
    Let $Y$ be a weakly geodesically convex closed subset of an $\RCD(K,\infty)$ space $(X, \dist,\meas)$ such that $\meas(Y)>0$ and $\meas(\partial_{\rm top}Y)=0$. Then $(Y,\dist_Y, \meas\mres Y)$ is also an $\RCD(K,\infty)$ space.
 \end{theorem}

More precisely, we conjecture that

\begin{conjecture}\label{conj:collapseconv}
    Let $\Omega$ be an open subset in a $\ncRCD(K,N)$ space $(X,\dist_X, \haus^N)$, where $N$ is a positive integer, so that for some Radon measure $\mu$ with $\supp \mu=\bar\Omega$, $(\bar{\Omega}, \dist_{\Omega},\mu)$ is an $\RCD(K,\infty)$ space. Assume that $\partial_{\rm top} \Omega$ is $\haus^{N-1}$-rectifiable, then $\mu\mres \Omega\ll \haus^N\mres\Omega$ and $\mu(\partial_{\rm top} \Omega)=0$ and every geodesic joining two points in $\Omega$ w.r.t. $\dist_\Omega$ does not intersect $\partial_{\rm top}\bar\Omega$ hence a local geodesic w.r.t. $\dist_X${, in particular, $\Omega$ is locally totally geodesic}.
\end{conjecture}

\section{ Almost convexity}\label{sec:almost}

\subsection{1-D localization}
We minimally collect the elements of the localization technique introduced in \cite{Cav14} and \cite{CavMon15}, we remark that this technique is available for a much general class of metric measure spaces, the so called essentially non-branching ${\rm MCP}(K,N)$ spaces, which contains essentially non-branching $\CD(K,N)$ spaces, hence $\RCD(K,N)$ spaces (\cite{RajalaSturm12}). 

Let $(X,\dist,\meas)$ be an $\RCD(K,N)$ space, $u$ be a $1$-Lipschitz function. Define the transport set induced by $u$ as:
\[
\Gamma(u)\defeq\{(x,y)\in X\times X: u(x)-u(y)=\dist(x,y)\},
\]
and its transpose as $\Gamma^{-1}(u)\defeq \{(x,y)\in X\times X: (y,x)\in \Gamma(u)\}$. The union $R_u\defeq \Gamma^{-1}(u)\cup \Gamma(u)$ defines a relation on $X$. By excluding negligible isolated and branching points, one can find a transport set $\mathcal{T}_u$ such that $\meas(X\setminus\mathcal{T}_u)=0$ and $R_u$ restricted to $\mathcal{T}_u$ is an equivalence relation. So there is a partition of $\mathcal{T}_u:=\cup_{\alpha\in Q} X_{\alpha}$, where $Q$ is a set of indices, denote by $\mathfrak{Q}:\mathcal{T}_u\to Q$ the quotient map. In \cite[Proposition 5.2]{Cav14}, it is shown that there exists a measurable selection $s:\mathcal{T}_u\to \mathcal{T}_u$ such that if $x R_u y$ then $s(x)=s(y)$, so we can identify $Q$ as $s(\mathcal{T}_u)\subset X$. Equip $Q$ with the $\sigma$-algebra induced by $\mathfrak{Q}$ and the measure $\mathfrak{q}\defeq \mathfrak{Q}_{\sharp}(\meas\mres\mathcal{T}_u)$, we can hence view $\mathfrak{q}$ as a Borel measure on $X$. Furthermore, each $X_{\alpha}$ is shown (\cite[Lemma 3.1]{CavMon15}) be to isometric to an interval $I_{\alpha}$, the distance preserving map $\gamma_{\alpha}: I_{\alpha}\to X_{\alpha}$ extend to an geodesic still denoted by $\gamma_{\alpha}:\bar{I}_{\alpha}\to X$. Putting several results together, we have (\cite[Theorem A.5]{KapMon19}):

\begin{theorem}\label{thm:disint}
    Let $(X,\dist,\meas)$ be an $\RCD(K,N)$ space. $u$ be a $1$-Lipschitz function. Then $\meas$ admits a disintegration:
    \[
    \meas=\int_{Q}\meas_{\alpha}\mathfrak{q}(\dd\alpha),
    \]
    where $\meas_{\alpha}$ is a non-negative Radon measure on $X$, such that 
    \begin{enumerate}
        \item For any $\meas$-measurable set $B$, the map $\alpha\mapsto \meas_{\alpha}(B)$ is $\mathfrak{q}$-measurable.
        \item\label{item:strcons} for $\mathfrak{q}$-a.e.\ $\alpha\in Q$, $\meas_{\alpha}$ is concentrated on $X_{\alpha}=\mathfrak{Q}^{-1}(\alpha)$. This property is called strong consistency of the disintegration.
        \item \label{item:disint} for any $\meas$-measurable set $B$ and $\mathfrak{q}$-measurable set $C$, it holds
        \[
        \meas(B\cap \mathfrak{Q}^{-1}(C))=\int_C \meas_{\alpha}(B)\mathfrak{q}(\dd\alpha).
        \]
        \item\label{item:pos} for $\mathfrak{q}$-a.e.\ $\alpha\in Q$, $\meas_{\alpha}=h_{\alpha}\haus^1\mres X_{\alpha}\ll\haus^1\mres X_{\alpha}$, where $h_{\alpha}$ is a $\log$ concave density, and $(\bar{X}_{\alpha}, \dist,\meas_{\alpha})$ is an $\RCD(K,N)$ space.
    \end{enumerate}
\end{theorem}

\subsection{ Proof of Proposition \ref{thm:almostconvex} and Consequences}

\begin{proof}[Proof of Proposition \ref{thm:almostconvex}]
Take $x\in X$, disintegrate $\meas$ w.r.t $\dist_x\defeq \dist(x,\cdot)$. Item \ref{item:disint} in Theorem \ref{thm:disint} yields that 
\begin{equation}
    \begin{split}
        0=\meas(X\setminus \mathcal{R}_n)=\int_Q \meas_\alpha(X\setminus \mathcal{R}_n)\mathfrak{q}(\dd\alpha).
    \end{split}
\end{equation}
Then for $\mathfrak{q}$-a.e.\ $\alpha\in Q$, $\meas_\alpha(X\setminus \mathcal{R}_n)=0$, we set $\widetilde{Q}\defeq \{\alpha\in Q: \meas_{\alpha}(X\setminus \mathcal{R}_n)=0 \}$, then $R_x\defeq (\cup_{\alpha\in \widetilde{Q}}X_{\alpha})\cap \mathcal{R}_n$ is the desired set. Indeed, for any $y\in R_x$, there is a geodesic (segment) $\gamma$ contained in $X_\alpha$ joining $x,y$, for some $\alpha\in \widetilde{Q}$, with $\meas_{\alpha}(\gamma\setminus \mathcal{R}_n)=0$.  the $\log$-concavity of $h_{\alpha}$ implies that $h_{\alpha}$ is $\mathcal{H}^1$ a.e.\ positive on $X_{\alpha}$, so we get that $\haus^1\mres X_\alpha(\gamma\setminus \mathcal{R}_n)=0$, which in turn implies that regular points of essential dimension is dense in the interior of $\gamma$. Now apply Proposition \ref{prop:dentofull}, we see that the interior of $\gamma$ is entirely in $\mathcal{R}_n$ and the end point $y$ is also in $\mathcal{R}_n$. 
\end{proof}


{ Since $\mathcal R_n\subset {\rm Int}(X)$ Proposition \ref{thm:almostconvex} immediately implies almost convexity of  ${\rm Int}(X) = X\setminus \partial X$. }


\begin{corollary}
Let $(X,\dist,\haus^N)$ be a $\ncRCD(K,N)$ space. For \textit{every} $x\in {\rm Int}(X)$,
 there exists a subset $R_x\subset {\rm Int}(X)$ so that $\meas(X\setminus R_x)=0$ and for any $y\in R_x$ there is a minimizing geodesic joining $x,y$ and entirely contained in ${\rm Int}(X)$. 

\end{corollary}

 
  We then naturally obtain the following corollary.

\begin{corollary}
In the setting of Theorem \ref{thm:main2}, for every point $x\in \Omega$, there exists a set $\mathcal{R}_x\subset\Omega$ such that $\haus^N(\bar\Omega\setminus\mathcal{R}_x)=0$ and for every $y\in\mathcal{R}_x$, there is a minimizing geodesic in $(\bar\Omega,\dist_\Omega)$ joining $x,y$ lies entirely in $\Omega$, hence a local geodesic in $(X,\dist_X)$, {i.e., $\Omega$ is almost locally totally geodesic}. 
\end{corollary}

\bibliographystyle{alpha}
\bibliography{convex}
\end{document}